\documentclass[11pt]{amsart}
\usepackage{amsfonts,latexsym,amsthm,amssymb,graphicx}
\usepackage[all]{xy}
\usepackage{hyperref}
\usepackage[usenames]{color} 
\usepackage{tikz}
\usetikzlibrary{shapes}
\usetikzlibrary{decorations.pathreplacing}

\DeclareFontFamily{OT1}{rsfs}{}
\DeclareFontShape{OT1}{rsfs}{n}{it}{<-> rsfs10}{}
\DeclareMathAlphabet{\mathscr}{OT1}{rsfs}{n}{it}

\newtheorem{theorem}{Theorem}[section]

{\theoremstyle{remark} \newtheorem{remark}[theorem]{Remark}
\newtheorem{example}[theorem]{Example}}

\newcommand{\Abb}{{\mathbb{A}}}
\newcommand{\Pbb}{{\mathbb{P}}}
\newcommand{\cA}{{\mathscr A}}
\newcommand{\cE}{{\mathscr E}}
\newcommand{\cF}{{\mathscr F}}
\newcommand{\cI}{{\mathscr I}}
\newcommand{\cK}{{\mathscr K}}
\newcommand{\cL}{{\mathscr L}}
\newcommand{\cO}{{\mathscr O}}
\newcommand{\cP}{{\mathscr P}}

\newcommand{\uF}{{\underline F}}
\newcommand{\uU}{{\underline U}}
\newcommand{\uX}{{\underline X}}
\newcommand{\uY}{{\underline Y}}
\newcommand{\Til}[1]{{\widetilde{#1}}}

\newcommand{\qede}{\hfill$\lrcorner$}
\DeclareMathOperator{\rk}{rk}
\DeclareMathOperator{\codim}{codim}
\DeclareMathOperator{\Sing}{Sing}

\title{
How many hypersurfaces does it take to cut out a Segre class?
}
\author{Paolo Aluffi}
\address{
Mathematics Department, 
Florida State University,
Tallahassee FL 32306, U.S.A.
}
\email{aluffi@math.fsu.edu}

\begin{document}

\begin{abstract}
We prove an identity of Segre classes for zero-schemes of compatible
sections of two vector bundles. Applications include bounds on the
number of equations needed to cut out a scheme with the same Segre
class as a given subscheme of (for example) a projective variety, and a
`Segre-Bertini' theorem controlling the behavior of Segre classes of
singularity subschemes of hypersurfaces under general hyperplane
sections.

These results interpolate between an observation of Samuel concerning
multiplicities along components of a subscheme and facts concerning
the integral closure of corresponding ideals. The Segre-Bertini theorem
has applications to characteristic classes of singular varieties.
The main results are motivated by the problem of computing Segre classes
explicitly and applications of Segre classes to enumerative geometry.
\end{abstract}

\maketitle

%%%

\section{Introduction}\label{sec:intro}

A result from P.~Samuel's thesis states that, under mild hypotheses,
in computing the multiplicity of a variety $Y$ along a subscheme $Z$
at an irreducible component $V$ of $Z$ we may replace the ideal 
determined by $Z$ in the local ring $\cO_{V,Y}$ by an ideal
generated by $\codim_VY$ elements (cf.~\cite[Theorem~22]{MR0120249}).  
In Fulton-MacPherson intersection
theory, the same multiplicity may be defined by means of {\em Segre
  classes\/} (\cite[\S4.3]{85k:14004}); it is then natural to ask
whether the number of equations needed to define a Segre class may be
similarly bounded. This is one of the questions we answer in this
note. We work over an algebraically closed field, and our schemes are
embeddable in nonsingular varieties.
We denote by $Z_\text{red}$ the reduced scheme supported on $Z$, and by
$s(Z,Y)$ the Segre class of $Z$ in $Y$.

\begin{theorem}\label{thm:main}
Let $Y$ be a pure-dimensional scheme, and let $Z\subseteq Y$ be a 
closed subscheme. Let $X_i$, $i=1,\dots $ be general elements of a linear
system cutting out $Z$.
\begin{itemize}
\item[$(a)$] Let $Z':= X_1\cap \cdots\cap X_{\dim Y+1}$.
Then ${Z'}_\text{red} = Z_\text{red}$, and $s(Z',Y)=s(Z,Y)$.
\item[$(b)$] Let $Z'':= X_1\cap \cdots\cap X_{\dim Y}$.
Then there exists an open neighborhood $Y^\circ$ of $Z$ in $Y$ such that 
$(Z''\cap Y^\circ)_\text{red} = Z_\text{red}$, and $s(Z''\cap Y^\circ,Y)=s(Z,Y)$.
\end{itemize}
\end{theorem}

\noindent (The equality of supports allows us to identify the relevant Chow groups, 
as required in order to compare the Segre classes, cf.~Remark~\ref{rem:supp}.)

Thus, the Segre class of $Z$ in $Y$ can be `cut out' by $\dim Y+1$
hypersurfaces, and by $\dim Y$ hypersurfaces in a neighborhood of $Z$. 
This fact is reminiscent of a well-known result of D.~Eisenbud and G.~Evans
(\cite{MR0327783}), stating that every subscheme $Z$ of $\Pbb^n$ may be 
cut out set-theoretically by $n$ hypersurfaces, and of an observation by 
W.~Fulton (\cite[Example 9.1.3]{85k:14004}) pointing out that $n+1$
hypersurfaces suffice to cut out $Z$ {\em scheme-theoretically\/} if $Z$ is 
locally a complete intersection. As a particular case of Theorem~\ref{thm:main}, 
$n+1$ hypersurfaces suffice to cut out a subscheme $Z'\subseteq \Pbb^n$ with 
the same Segre class in $\Pbb^n$ as $Z$, without any requirement on~$Z$.  
These hypersurfaces may be chosen to be general in a linear 
system cutting out $Z$, and $n$ hypersurfaces suffice in a neighborhood 
of $Z$.\smallskip

Theorem~\ref{thm:main} may be further refined, as follows. Denote by $s(Z,Y)_k$ the
$k$-dimensional component of the Segre class $s(Z,Y)$.

\addtocounter{theorem}{-1}
\begin{theorem}
{\rm (continued)}
\begin{itemize}
\item[$(b')$] More generally, let $c\ge 0$ and let $Z_{(c)}:=X_1\cap \cdots \cap X_{\dim Y-c}$. 
Then there exists a closed subscheme $S$ of dimension $\le c$ in $Y$
such that $\dim (S\cap Z)<c$, 
$(Z_{(c)}\smallsetminus S)_\text{red} = (Z\smallsetminus S)_\text{red}$,
$s(Z_{(c)}\smallsetminus S,Y\smallsetminus S)_c
=s(Z\smallsetminus S, Y\smallsetminus S)_c$, and
$s(Z_{(c)},Y)_k=s(Z,Y)_k$ for $k> c$. 
\end{itemize}
\end{theorem}

For $c=0$, part $(b')$ of Theorem~\ref{thm:main} reduces to part $(b)$: in this case $S$
is a set of points disjoint from $Z$, and we can take $Y^\circ = Y\smallsetminus S$.
Part $(a)$ may also formally be seen as a particular case of $(b')$, by allowing $c=-1$ 
(and hence $S=\emptyset$).

If $V$ is an irreducible component of $Z$, it follows from Theorem~\ref{thm:main}$(b')$
with $c=\dim V$
that the coefficient of $V$ in $s(Z,Y)$ equals the coefficient of $V$ in $s(Z_{(\dim V)},Y)$.
As this coefficient equals the multiplicity of $Y$ along $Z$ at $V$ 
(\cite[Example~4.3.4]{85k:14004}), this recovers Samuel's result.\smallskip

Theorem~\ref{thm:main} is motivated by effective computations of Segre
classes and of contributions of components to an intersection.

\begin{example}\label{exa:linsys}
Consider the scheme $Z\subseteq \Pbb^3$ defined by the ideal
\[
(z^2,yz,xz,y^2w-x^2(x+w))\quad.
\]
(This is the flat limit of a family of twisted cubics,
cf.~\cite[Example~9.8.4]{MR0463157}.)  Note that $Z$ may be cut out by
cubics. Standard techniques give $s(Z,\Pbb^3) =[Z]-10[pt]$. If
$X_1,X_2,X_3$ are general cubic hypersurfaces containing $Z$, then by
\cite[Proposition~9.1.1]{85k:14004} the contribution of $Z$ to the
intersection $X_1,X_2,X_3$ is
\begin{equation}\label{eq:exa}
\int c(\cO(3H))^3\cap s(Z^-,\Pbb^3)
\end{equation}
where $H$ denotes the hyperplane class and $Z^-$ is the component of
$X_1\cap X_2\cap X_3$ supported on $Z$. (So $Z^-=Z''\cap Y^\circ$ with
the notation of Theorem~\ref{thm:main}$(b)$.) A Macaulay2 (\cite{M2})
computation shows that the schemes $Z$ and $Z^-$ have the same support
but are not equal. In fact, the scheme $Z^-$ depends on the choice of $X_1,
X_2,X_3$, so it seems {\em a priori\/} difficult to perform the computation 
of~\eqref{eq:exa}, barring an exhaustive analysis of the specific choice of these
hypersurfaces. However, by Theorem~\ref{thm:main}$(b)$ we must have
\[
s(Z^-,\Pbb^3)=s(Z,\Pbb^3) = [Z]-10[pt]\quad,
\]
and it follows that the contribution computed by~\eqref{eq:exa} is $\int (1+9H)\cap 
([Z]-10[pt])= 17$.  Taking this off the B\'ezout number $27$ for the intersection of 
three cubics, it follows that $X_1,X_2,X_3$ meet at $10$ points outside of $Z$ 
(if the ground field is algebraically closed of characteristic~$0$).
\qede\end{example}

The point of this example is that even though the relevant component $Z^-$
of the intersection of the hypersurfaces is {\em not\/} equal to $Z$,
we may carry out the computation of the corresponding contribution to
the intersection as if it were. This is a common issue in applications
of Segre classes to enumerative geometry, where $Z$ may have a compelling
scheme-theoretic description, but the scheme $Z^-$ corresponding to a 
choice of hypersurfaces realizing general constraints may retain features
due to the specific chosen hypersurfaces. Example~\ref{exa:linsys} illustrates 
the fact that the hypothesis that the expected number of divisors in the linear 
system cut out the base scheme in a neighborhood 
(cf.~\cite[Example~4.4.1]{85k:14004}) may be weakened: in general this 
hypothesis should not be expected to be verified, but the corresponding 
formulas remain true if the divisors are general. We should also point out that,
as a rule, judicious use of~\cite[Proposition~4.4]{85k:14004} suffices to 
address this issue; in fact, Theorem~\ref{thm:main}$(b)$ is little more than a 
recasting of this result, presented here in an attempt to streamline its utilization. 

The same complication arises in some approaches to the algorithmic computation of
Segre classes. For example, the computation of the Segre class of a
subscheme in $\Pbb^n$ is reduced in~\cite{EJP} to residual
intersection computations, and these are controlled by the Segre class
of an {\em a priori\/} different subscheme; again,
Example~\ref{exa:linsys} provides an example.  Theorem~3.2
in~\cite{EJP} implicitly includes a proof of Theorem~\ref{thm:main}$(b)$ in the 
particular case $Y=\Pbb^n$, resolving this issue in this case for the numerical 
degrees of the classes (i.e., their push-forward to projective space).
Theorem~2 in~\cite{MR3101819} does the same in the toric setting.
Theorem~\ref{thm:main} has no restrictions on the ambient scheme $Y$, and 
gives the result in the Chow group of $Z$ rather than pushing-forward to $Y$.
Also, Theorem~\ref{thm:main}($b'$) can in principle lead to an improvement in the
speed of such algorithms when only terms of a fixed dimension in the Segre
class are needed.\smallskip

We prove Theorem~\ref{thm:main} in~\S\ref{sec:proof1} as an 
application of a more general observation presented in~\S\ref{sec:tool}, 
to the effect that the Segre class of the zero-scheme of a section of a 
vector bundle $\cE$ is preserved in a range of dimensions by taking 
suitable quotients of $\cE$. See~\S\ref{sec:tool} for a precise statement.  
For the `$c=0$ case', corresponding to part $(b)$ of Theorem~\ref{thm:main},
the commutative algebra 
counterpart of this observation is the statement that such quotients do not 
change the integral closure of the ideal sheaf determined by the section. 
The argument given in~\S\ref{sec:tool} to prove the Segre class identity 
may be used to draw this conclusion (Remark~\ref{rem:ic}); but our proof 
of the Segre class identity bypasses the commutative algebra, and hence 
seems more direct in the context of this paper.

In~\S\ref{sec:SB} we give a second application of the same tool,
proving a `Bertini' type statement for Segre classes; here we require
the characteristic of the field to be~$0$.  Let $Y\subseteq \Pbb^n$
now be a nonsingular variety, and let $X$ be a (possibly 
singular) hypersurface in $Y$. Let $H$ be a general hyperplane. 
By the Bertini theorem, the singular locus of $H\cap X$ is set-theoretically 
equal to the intersection of $H$ with the singular locus of $X$:
\[
(\Sing(H\cap X))_\text{red} = (H\cap \Sing(X))_\text{red}
\]
where $\Sing(-)$ denotes the singularity subscheme. While this equality is not
true at the level of schemes, we prove that it does lift to an equality of Segre classes.

\begin{theorem}\label{thm:SB}
Let $H$ be a general hyperplane, and let $W$ be a pure-dimensional
subscheme of $H\cap Y$. Then $s(\Sing (H\cap X)\cap W, W) = s(H\cap
\Sing(X)\cap W,W)$.
\end{theorem}

Again this statement could be proven from commutative algebra 
considerations. The argument given here follows easily from the tool
presented in~\S\ref{sec:tool}, and seems more direct.
Theorem~\ref{thm:SB} may be used to prove that certain characteristic
classes associated with $X$ behave as expected with respect to general
hyperplane sections.  There are other approaches to such questions;
see for example \cite[Lemma~1.2]{MR1311826}.
\smallskip

{\em Acknowledgments.}  The author's research was supported in part by
the Simons foundation and by NSA grants H98230-15-1-0027 and 
H98230-16-1-0016. The author is grateful to Caltech for hospitality while 
this work was carried out. The author especially thanks Prof.~Matilde Marcolli for
suggesting the title of this paper. The author also thanks Sam Huckaba
for insightful conversations on the commutative algebraic aspects of the 
questions studied in this paper.

%%%

\section{The main tool}\label{sec:tool}

In this note $Y$ denotes a separated pure-dimensional scheme of finite type 
over an algebraically closed field, embeddable in a nonsingular scheme.
We are interested in the Segre classes $s(Z,Y)$ of subschemes $Z$ of
$Y$. By~\cite[Lemma~4.2]{85k:14004} the Segre class of $Z$ in $Y$ is a
linear combination of the Segre classes in the irreducible components
of $Y$; so we may and will assume that $Y$ is a variety. Also,
$s(Y,Y)=[Y]$; so we will assume the subschemes we consider are
properly contained in $Y$.

Every subscheme $Z$ of $Y$ may be realized as the zero-scheme of a
section of a vector bundle $\cE$ on $Y$, and this yields an embedding
of the blow-up $B\ell_ZY$ of $Y$ along $Z$ as a subscheme of
$\Pbb(\cE)$ (\cite[B.8.2]{85k:14004}). We consider the following
situation:
\begin{itemize}
\item $\cE$, $\cF$ are vector bundles on $Y$;
\item $s_\cE$, $s_\cF$ are sections of $\cE$, $\cF$, with zero-schemes
  $Z$, $Z'$, respectively;
\item We have an epimorphism $p: \cE\to \cF$ with kernel $\cK=\ker p$,
  such that the diagram
\[
\xymatrix{
\cE \ar[rr]^p & & \cF \\
& Y \ar[ul]^{s_\cE} \ar[ur]_{s_\cF}
}
\]
is commutative;
\end{itemize}

Clearly $Z\subseteq Z'$, in the sense that the ideal sheaf $\cI_{Z,Y}$
of $Z$ in $Y$ contains $\cI_{Z',Y}$.

As recalled above, the blow-up $B\ell_ZY$ may be embedded in the
projectivization $\Pbb(\cE)$. The image of this embedding is the
variety $\Til Y$ obtained as the closure of the image in $\Pbb(\cE)$ of
the rational section $Y \dashrightarrow \Pbb(\cE)$ induced by
$s_\cE$. We will denote by $E\subseteq \Til Y$ the exceptional divisor.
We adopt the convention that the empty set has negative dimension.

\begin{theorem}\label{thm:tool}
Let $c=\dim (\Pbb(\cK)\cap \Til Y)$ and assume that no component of
$\Pbb(\cK)\cap \Til Y$ is contained in $E$.
Then there exists a 
subscheme $S\subseteq Y$ of dimension $\le c$ such that 
$\dim (S\cap Z)<c$ and
\begin{itemize}
\item[$(i)$] $(Z'\smallsetminus S)_\text{red} = (Z\smallsetminus S)_\text{red}$;
\item[$(ii)$] $s(Z'\smallsetminus S,Y\smallsetminus S)_c
=s(Z\smallsetminus S, Y\smallsetminus S)_c$; 
\item[$(iii)$] $s(Z',Y)_k=s(Z,Y)_k$ for $k> c$.
\end{itemize}
\end{theorem}

\begin{remark}\label{rem:supp}
The equality of supports in $(i)$ allows us to identify the Chow groups 
as needed for $(ii)$ and $(iii)$. Indeed, if $(i)$ holds, then $A_c(Z'\smallsetminus S)
=A_c(Z\smallsetminus S)$ (\cite[Example 1.3.1]{85k:14004}); and recall
(\cite[Proposition~1.8]{85k:14004}) that for all $k\ge 0$ there is an exact sequence
\[
\xymatrix{
A_k(Z\cap S) \ar[r] & A_k(Z) \ar[r] & A_k(Z\smallsetminus S) \ar[r] & 0\quad;
}
\]
if $\dim (S)\le c$ and $(i)$ holds, then it follows that for $k> c$ we have canonical isomorphisms
\[
A_k(Z) \cong A_k(Z\smallsetminus S) \cong A_k((Z\smallsetminus S)_\text{red}) 
\cong A_k((Z'\smallsetminus S)_\text{red}) \cong A_k(Z'\smallsetminus S) \cong A_k(Z')\quad.
\]
The claimed equality $s(Z,Y)_k=s(Z',Y)_k$ holds in this group.
\qede
\end{remark}

\begin{example}
It is important to note that $Z\ne Z'$ in general, even if $\Pbb(\cK)\cap \Til Y=\emptyset$. 
For example, let
$Y=\Abb^2$ with coordinates $x,y$, and let $\cE=\cO^{\oplus 3}$,
projecting to the first two factors $\cF=\cO^{\oplus 2}$. We have
$\cK\cong \cO$, identified with the third factor of $\cE$. Define
$s_\cE$ by
\[
(x,y) \mapsto (x^2,y^2,xy)\quad.
\]
It is straightforward to verify that $\Til Y$ is given by the ideal
\[
(sy-ux,tx-uy,st-u^2)
\]
in $\Pbb(\cE)\cong \Abb^2\times \Pbb^2$, where $(s:t:u)$ are
homogeneous components in the $\Pbb^2$ factor. The projectivization
$\Pbb(\cK)$ consists of $\Abb^2\times \{(0:0:1)\}$, hence it has empty
intersection with $\Til Y$. On the other hand, the ideals of $Z$, $Z'$
are $(x^2,y^2,xy)$, $(x^2,y^2)$, respectively; so $Z\ne Z'$.
\qede\end{example}

\begin{proof}[Proof of Theorem~\ref{thm:tool}]
Identify $\Til Y$ with $B\ell_ZY$; the blow-up map is the projection 
$\pi: \Til Y \to Y$, and $E=\pi^{-1}(Z)$ is the exceptional divisor. 
Since $Z\subseteq Z'$, we have $E\subseteq \pi^{-1}(Z')$. We will 
verify that the residual scheme to $E$ in $\pi^{-1}(Z')$ is supported 
on $S':=\Pbb(\cK)\cap \Til Y$. We claim that the statement of the 
theorem follows from this, by setting $S=\pi(S')$. Indeed, since no 
component of $S'$ is contained in $E$, we have $\dim (Z\cap S)<c$. 
Since (set-theoretically) $Z'=\pi(\pi^{-1}(Z'))=\pi(E\cup S') =Z\cup S$,
the equality $(i)$ of supports holds. Further, we will have
\[
s(Z',Y)_k\overset{(1)}=\pi_* s(\pi^{-1}(Z'),\Til Y)_k \overset{(2)}= \pi_* s(E,\Til Y)_k 
\overset{(3)}= s(Z,Y)_k
\]
for $k> c$, 
where equalities $(1)$ and $(3)$ hold by the birational invariance of Segre
classes (\cite[Proposition~4.2(a)]{85k:14004}), and $(2)$ follows
from the residual formula for Segre classes (\cite[Proposition~9.2]{85k:14004}).
This implies $(iii)$. The residual formula also
shows that $s(E,\Til Y)_c$ and $s(\pi^{-1}(Z),\Til Y)_c$ differ
by a class supported on $S'$. It follows that 
$s(E\smallsetminus S',\Til Y\smallsetminus S')_c
=s(\pi^{-1}(Z)\smallsetminus S',\Til Y\smallsetminus S')_c$, 
and $(ii)$ follows, again by the birational invariance of Segre classes.

Thus, in order to prove the theorem it suffices to show that the residual 
scheme to $E$ in $\pi^{-1}(Z')$ equals $S':=\Pbb(\cK)\cap \Til Y$.
To compute this residual scheme we may work locally, hence
assume that $\cE$, $\cK$ and $\cF$ are trivial and that 
$p:\cE=\cO^{\oplus N}\to \cF=\cO^{\oplus M}$ is the projection onto the first 
$M$ factors. 
Write the section $s_\cE$ in components as
$s_\cE=(s_1,\dots,s_N)$; the induced rational section
$Y\dashrightarrow \Pbb(\cE)=Y\times \Pbb^{N-1}$ is $(s_1:\cdots :
s_N)$, and this lifts to the embedding
\[
\rho: \Til Y \to Y\times \Pbb^{N-1}\quad.
\] 
The exceptional divisor $E$ is given by a section $e$ of $\cO(E)$. 
Locally, $E$ is defined by the ideal $(e)=(\pi^{-1} s_1,\dots, \pi^{-1} s_N)$. 
We can factor
\[
\pi^{-1}(s_i) = \hat s_i e
\]
for $i=1,\dots, N$; then $\hat s_1,\dots, \hat s_N$ locally generate 
the unit ideal $(1)$, and the embedding $\rho: \Til Y \to Y\times
\Pbb^{N-1}$ is given by 
$\tilde y \mapsto (\tilde y,(\hat s_1(\tilde y):\cdots: \hat s_N(\tilde y)))$.

With the above notation, the ideal for $S'=\Pbb(\cK)\cap \Til Y$ in $\Til Y$ is
$\cI_{S',\Til Y}=(\hat s_1,\dots, \hat s_M)$.

Now $Z'$ is defined by the ideal $(s_1,\dots, s_M)$ in $Y$. Therefore
$\pi^{-1}(Z')$ has ideal
\[
(\hat s_1 e,\dots, \hat s_M e) = (\hat s_1,\dots, \hat s_M)\cdot (e) = 
\cI_{S',\Til Y}\cdot \cI_{E,\Til Y}\quad.
\]
This verifies that the residual scheme to $E$ in $\Til Y$ is 
$S'=\Pbb(\cK)\cap \Til Y$ and concludes the proof.
\end{proof}

\begin{remark}\label{rem:Samuel}
Suppose $c=\dim V$, where $V$ is an irreducible component of $Z$.
By Theorem~\ref{thm:tool}$(i)$, $V$ is also an irreducible component
of $Z'$, and by Theorem~\ref{thm:tool}$(ii)$, $V$ appears with the
same multiplicity in $s(Z,Y)$ and in $s(Z', Y)$. Thus the multiplicity
of $Y$ along $Z$ and along $Z'$ at $V$ coincide. This will 
recover the result by Samuel recalled at the beginning of~\S\ref{sec:intro}.
\qede\end{remark}

\begin{remark}\label{rem:ic}
If $c=0$ in Theorem~\ref{thm:tool}, then $S'=\Pbb(\cK)\cap \Til Y$ consists 
of a finite set disjoint from $E$; replacing $Y$ by the complement of $S=\pi(S')$, 
we may assume $\Pbb(\cK)\cap \Til Y=\emptyset$. The resulting situation
has a compelling interpretation in terms of commutative algebra. We have a 
commutative diagram of rational maps
\[
\xymatrix{
\Pbb(\cE) \ar@{-->}[rr]^p & & \Pbb(\cF) \\
& Y \ar@{-->}[ul]^{s_\cE} \ar@{-->}[ur]_{s_\cF}
}
\]
and the projection induces a birational morphism between the closures
of the images of $s_\cE$ and $s_\cF$, i.e.,
\[
\xymatrix{
p|_Y\,: \, \Til Y= B\ell_Z Y \ar@{-->}[r] & B\ell_{Z'} Y\quad.
}
\]
If $\Pbb(\cK)\cap \Til Y=\emptyset$, this morphism is
regular and finite, because $\Til Y$ is disjoint from the center of the projection.
This implies that the ideal of $Z'$ is a reduction of the ideal of $Z$ at every
point of $Z$, cf.~\cite[Proposition~1.44]{MR2153889}. Therefore, we can 
conclude that if $c=0$ in Theorem~\ref{thm:tool}, then the ideals of $Z$ and
$Z'$ have the same integral closure at every point of~$Z$.
Theorem~\ref{thm:tool} may be seen as a Segre class version of this observation,
extended to all~$c\ge 0$.
\qede\end{remark}

%%%

\section{Proof of Theorem~\ref{thm:main}}\label{sec:proof1}

As above, $Y$ is a pure-dimensional scheme and $Z\subseteq Y$ is
a closed subscheme, and we can in fact assume that $Y$ is a variety and
$Z\subsetneq Y$ (cf.~the beginning of~\S\ref{sec:tool}).

Let $\cA$ be a line bundle on $Y$, and let $L\subseteq \Pbb H^0(\cA,Y)$ be a 
linear system of which $Z$ is the base 
scheme. For example, if $Y\subseteq \Pbb^n$, then $L$ can be the restriction 
of the linear system of degree-$d$ hypersurfaces containing~$Z$, provided 
$d$ is large enough; in fact, the maximum degree in a set of generators of any 
ideal defining $Z$ scheme-theoretically in~$\Pbb^n$ will do. 
If $X_1,\dots,X_N$ are general elements of~$L$ and $N\gg 0$, then $X_1,
\dots, X_N$ generate~$L$, and $Z=X_1\cap \cdots \cap X_N$ scheme-theoretically.
Equivalently, $Z$ may be realized as the zero-scheme of the section
$s_\cE$ of $\cE:=\cA^{\oplus N}$ given by $(s_1,\dots,s_N)$, where $s_i$
is a section of~$\cA$ defining $X_i$.  As in~\S\ref{sec:tool}, we consider the 
closure $\Til Y$ of the image of the rational section $(s_1:\dots: s_N)$ of 
$\Pbb(\cE)$ determined by~$s_\cE$. For $c\ge 0$, we let $M=\dim Y-c$, and
we let $Z_{(c)}=X_1\cap \cdots \cap X_{\dim Y-c}$ be the zero-scheme of the 
section $s_\cF$ of $\cF:=\cA^{\oplus M}$ given by $(s_1,\dots, s_M)$. 

We are then in the situation of \S\ref{sec:tool}. Let $\cK$ be the kernel of the 
projection $p:\cE\to \cF$. By Theorem~\ref{thm:tool}, in order to prove 
Theorem~\ref{thm:main} it suffices to show that $\Pbb(\cK)\cap \Til Y$ has
pure dimension $c$ (if nonempty) and $\Pbb(\cK)\cap E$ has dimension $<c$.
 We have
\[
\Pbb(\cE)=\Pbb(\cA^{\oplus N})\cong \Pbb(\cO_Y^{\oplus N}) =
Y\times \Pbb^{N-1}\quad;
\]
$\Pbb(\cF)\cong Y\times \Pbb^{M-1}$ under the same identification, and
$\Pbb(\cK)$ is defined by the intersection of $M$ hypersurfaces
$Y\times H^{(1)},\dots,Y\times H^{(M)}$, where $H^{(i)}$ are general
hyperplanes. Now the linear system cut out on $\Til Y$ by the hypersurfaces
$Y\times H$, with $H$ a hyperplane, is base-point free. It follows that 
every component of 
$\Pbb(\cK)\cap \Til Y$ has codimension $M=\dim Y-c$ in $\Til Y$, 
hence dimension~$c$. 
The dimension of every component of $\Pbb(\cK)\cap E$ is $c-1$ by
the same token, concluding the proof.
\qed\smallskip

If $V$ is an irreducible component of $Z$, the $c=\dim V$ case of 
Theorem~\ref{thm:main}$(b')$ recovers Samuel's result on multiplicities,
cf.~Remark~\ref{rem:Samuel}.

As a consequence of the considerations in Remark~\ref{rem:ic}, we see that 
for $c=0$ this argument in fact proves 
that the local equations of $\dim Y$ general elements of $L$ generate a 
reduction of the ideal $\cI_{Z,Y}$ of $Z$ in $Y$ at every point $p\in Z$. 
For example, if $Y\subseteq \Pbb^n$, $\dim Y$ general homogeneous 
polynomials of degree~$d$ containing $Z$ (where $d\ge $ maximum degree 
of a polynomial in any set of generators of any ideal for $Z$ in $\Pbb^n$) 
generate a reduction 
of $\cI_{Z,Y}$ at every point of $Z$. In the local case, it is known that a 
reduction may in fact be generated by $\ell$ generators, where $\ell$ equals 
the {\em analytic spread\/} of the ideal (\cite[Proposition~8.3.7]{MR2266432}).  
In our context, this equals $1$ plus the dimension of the fiber of the 
exceptional divisor $E$ at the given point. Theorem~\ref{thm:main}$(b)$ 
amounts to the consequence for Segre classes of a global version of this 
result.

%%%

\newpage
\section{A Segre-Bertini theorem}\label{sec:SB}

We now move on to the proof of Theorem~\ref{thm:SB}. In this section
we assume that the characteristic of the ground field is~$0$.

We consider a nonsingular variety $Y\subseteq \Pbb^n$ and let
$X\subseteq Y$ be a hypersurface.  We denote by $Z=\Sing(X)$ the {\em
  singularity subscheme\/} of $X$, i.e., the subscheme of~$Y$ locally
defined by an equation for $X$ and by its partial derivatives.
The Segre classes of $Z$ in $X$ and in $Y$ play an important role in
the theory of characteristic classes for singular varieties: there are
formulas relating directly the class $s(Z,X)$ with the {\em
  Chern-Mather\/} class of $X$ and the class $s(Z,Y)$ with the {\em
  Chern-Schwartz-MacPherson\/} class of $X$ (see e.g.,
\cite[Proposition~2.2]{MR2020555}).

By the ordinary Bertini theorem $\uY:=H\cap Y$ is nonsingular for a
general hyperplane $H$, and the singularity subscheme of $\uX:=H\cap
X$ is supported on $\underline{\Sing(X)}:=H\cap Z$.  It is clear that
$\underline{\Sing(X)}\subseteq \Sing(\uX)$, but these two subschemes
may be different, even for general~$H$.
According to Theorem~\ref{thm:SB}, this difference does not affect
their Segre classes: we will prove that if $W$ is any pure-dimensional
subscheme of~$\uY$, then
\begin{equation}\label{eq:SB}
s(\Sing(\uX)\cap W,W) = s(\underline{\Sing(X)}\cap W,W)\quad.
\end{equation}
For example, this holds for $W=\uX$ and $W=\uY$, the cases most
relevant to characteristic classes as mentioned above. Also, since
$\uY$ is nonsingular, the same equality follows for any nonsingular
variety $W\subseteq \Pbb^n$ containing $\Sing(\uX)$ (by
\cite[Example~4.2.6(a)]{85k:14004}).

The main value of identities such as~\eqref{eq:SB} is that they yield
tools for the effective computation of Segre classes. For example, for
$W=\uY$, \eqref{eq:SB} implies that
\[
s(\Sing(\uX),\uY)=H\cdot s(\Sing(X),Y)
\]
by essentially the same argument used in the proof of~\cite[Claim
  3.2]{MR3076849},
and this implies an `adjunction formula' for Chern-Schwartz-MacPherson
classes (cf.~\cite[Proposition~2.6]{MR3031565}).

Like Theorem~\ref{thm:main}, Theorem~\ref{thm:SB} (i.e.,
\eqref{eq:SB}) follows from Theorem~\ref{thm:tool}: it suffices to
realize the two schemes as zero-schemes of compatible sections of
vector bundles under a projection from a suitable center. The main 
technical point of the proof is the existence of such a center.

\begin{proof}
As in the proof of Theorem~\ref{thm:main} we may assume that $W$ is a
variety.  Let $\cL=\cO(X)$; so $X$ is defined by a section $F$ of
$\cL$ on $Y$. This section lifts to a section $s_F$ of the bundle of
principal parts $\cP^1_Y\cL$, and the subscheme $\Sing(X)$ is the
zero-scheme of this section. Therefore, $\underline{\Sing(X)}$ is the
zero-scheme of the restriction of $s_F$ to $\uY$, a section of
$(\cP^1_Y\cL)|_\uY$.  By the same token, $\Sing(\uX)$ is the
zero-scheme of the section $s_\uF$ of $\cP^1_\uY\cL$ determined by the
restriction $\uF$ of $F$ to $\uY$ (where for brevity we denote by $\cL$
the restriction $\cL|_\uY$).  These sections are compatible with
the natural surjective morphism of vector bundles $(\cP^1_Y\cL)|_\uY
\to \cP^1_\uY\cL$: the diagram
\[
\xymatrix{
(\cP^1_Y\cL)|_\uY \ar[rr]^p & &  \cP^1_\uY\cL \\
& \uY \ar[lu]^{(s_F)|_\uY} \ar[ru]_{s_\uF}
}
\]
is commutative. We are therefore in the situation studied in~\S\ref{sec:tool}, 
and in order to complete the proof we only need to verify that, for a general 
$H$, the projectivization of the kernel of~$p$ is disjoint
from the closure of the image of $s_F(W)$ in $\Pbb((\cP^1_Y\cL)|_W)$.
{\em A fortiori,\/} it suffices to show that this is the case for $W=\uY$.
This will verify the hypothesis of Theorem~\ref{thm:tool} with $c=-1$,
hence prove the equality of Segre class in all dimensions.

Recall that, by \cite[16.4.20]{MR0238860}, there is an exact sequence
\[
\xymatrix{
0 \ar[r] & \cL\otimes \cO(-1)|_\uY \ar[r] & (\cP^1_Y\cL)|_\uY \ar[r]^-p & \cP^1_\uY\cL \ar[r] & 0
}
\]
extending the standard exact sequence of differentials from
\cite[Proposition~II.8.12]{MR0463157},
\[
\xymatrix{
0 \ar[r] & \cO(-1)|_\uY \ar[r] & {\Omega_Y}|_\uY \ar[r] & \Omega_\uY \ar[r] & 0
}
\]
tensored by~$\cL$. (This sequence is exact on the left since $Y$ and
$\uY$ are nonsingular, \cite[Proposition~II.8.17]{MR0463157}.)
Therefore, the kernel $\cK$ of $p$ is the image of $\cL\otimes
\cO(-1)|_\uY$ in $(\cP^1_Y\cL)|_\uY$, and $p:
\Pbb((\cP^1_Y\cL)|_\uY)\dashrightarrow \Pbb(\cP^1_\uY\cL)$ is the
projection with center at the section $\Pbb(\cK)$.  By Theorem~\ref{thm:tool},
in order to prove Theorem~\ref{thm:SB} it suffices to prove that, for a
general choice of~$H$, $s_H=\Pbb(\cL\otimes \cO(-1)|_\uY)$ is disjoint
from $B\ell_{\underline{\Sing(X)}}\uY$ in
$\Pbb((\cP^1_Y\cL)|_\uY)$.\smallskip

To study this question, it is helpful to work over a trivializing open
set. (A local trivialization for the bundle of principal parts is
discussed in e.g., \cite[\S{A}.4]{MR96c:14023}.) Let $U\subseteq Y$ be
a dense open set such that
\begin{equation}\label{eq:triv}
\Pbb(\cP^1_U \cL) \cong U\times \Pbb^m
\end{equation}
with $m=\rk \Pbb(\cP^1_U\cL) = \dim Y$. We may choose $U$ so that the
projection from a fixed subspace $P:=\Pbb^{n-m-1}$ is an isomorphism
on each embedded tangent space to $Y$ at $y\in U$; we then have a
natural identification of the fiber $\Pbb^m$ in~\eqref{eq:triv} with
the subspace of the dual space ${\Pbb^n}^\vee$ consisting of
hyperplanes containing $P$.

We will denote by $(y,H)$ the point of $\cP^1_U\cL$ determined by the
choice of a point $y\in U$ and a hyperplane $H\supseteq P$.

With this notation, $(y,H)\in \Pbb(\Omega^1_U\otimes \cL)\subseteq
\Pbb(\cP^1_U\cL)$ if and only if $y\in H$.
Further, each $H\supseteq P$ determines a section of $\Pbb^1(\cP^1_U\cL)$,
given by $y \mapsto (y,H)$ for $y\in U$, and hence a section of $\Pbb((\Omega^1_U
\otimes\cL)|_\uU)$ for $\uU=H\cap U\subseteq \uY$.  It is straightforward to
verify that this section agrees with the restriction to $U$ of the
section $s_H$ determined by $H$ as explained above.

We have to prove that, for a general hyperplane $H$, $s_H$ is disjoint
from $B\ell_{\underline{\Sing(X)}}\uY$ in $\Pbb((\cP^1_Y\cL)|_\uY)$. 
It suffices to prove that there is {\em one\/} such hyperplane.
Arguing by contradiction, assume that for all $H$ there exists a point
$y\in Y$ such that $s_H$ and $B\ell_{\underline{\Sing(X)}}\uY$ meet
over $y$. After a choice of $U$ and $P$ as above, we may represent
points of $\Pbb(\cP^1_U\cL)$ by pairs $(y,H)$ with $y\in U$ and
$H\supseteq P$. By our assumption, we would have that for a general
$H$ containing $P$ there exists $y\in \uU$ such that
\begin{equation}\label{eq:yH}
(y,H)\in B\ell_{\underline{\Sing(X)}}\uY \subseteq B\ell_{\Sing(X)} Y\quad.
\end{equation}
The set of such $(y,H)\in B\ell_{\Sing(X)} Y$ is $m$-dimensional, since it 
dominates the fiber $\Pbb^m$ in the trivialization~\eqref{eq:triv}.  
Since $\dim B\ell_{\Sing(X)} Y=m$, it follows that {\em every\/} $(y,H)\in
B\ell_{\Sing(X)} Y$ is of this type. But $(y,H)\in
\Pbb(\Omega^1_U\otimes \cL)$ since $y\in \uU\subseteq H$. Thus, it
would follow that
\[
B\ell_{\Sing(X)} Y\subseteq \Pbb(\Omega^1_Y\otimes \cL)\quad.
\]
However, this is clearly not the case: the fiber of $B\ell_{\Sing(X)}
Y$ over a point $y\not\in X$ equals $s_F(y)$, which is not an element
of the fiber of $\Pbb(\Omega^1_Y\otimes \cL)$ since $F(y) \ne 0$.

This contradiction concludes the proof of Theorem~\ref{thm:SB}.
\end{proof}

As in Remark~\ref{rem:ic}, we can also observe that this argument
proves that for a general hyperplane $H$ there is a regular finite map
\begin{equation}\label{eq:rm}
B\ell_{\underline{\Sing(X)}}Y \to B\ell_{\Sing(\uX)}Y
\end{equation}
(restricting to a finite regular map $B\ell_{\underline{\Sing(X)}\cap
  W}W \to B\ell_{\Sing(\uX)\cap W} W$ for all $W\subseteq Y$).  It
follows that the ideal of $\underline{\Sing(X)}$ is integral over the
ideal of $\Sing(\uX)$ for a general $H$. This (re)proves a particular
case of Teissier's `idealistic Bertini theorem',
\cite[\S2.8]{MR58:27964}. In fact, the idealistic Bertini can conversely be
used to prove Theorem~\ref{thm:SB}.

%%%

\end{document}